\documentclass[12pt,leqno]{article}

\usepackage{time}
\usepackage{graphicx}
\usepackage{multicol}
\usepackage{amsmath}
\usepackage{amssymb}
\usepackage{amsfonts}
\usepackage{latexsym}
\usepackage[T1]{fontenc}
\usepackage{subfigure}
\usepackage{float}
\usepackage{mathrsfs}
\usepackage{epstopdf}

\newtheorem{thm}{Theorem}
\newtheorem{lem}{Lemma}

\newcommand{\dm}[1]{{\displaystyle{#1}}}
\newcommand{\integ}[4]{\dm{\int_{#1}^{#2}{#3}{\,d\,}{\!#4}}}
\newcommand{\msf}[1]{{\sf #1}}

\newenvironment{proof}
{\begin{sloppypar}\noindent{\sf Proof: }}
	{~\eop
\medskip
\end{sloppypar}}

\def\eop{\mbox{\quad{\vrule height7pt width3pt depth0pt}}}
\def\ilbr{{{[\![}}}
\def\irbr{{{\,]\!]}}}
\def\lbr{{[}}
\def\rbr{{]}}
\def\RR{\mathbb{R}}
\def\F{\mathcal{F}}
\def\I{\mathcal{I}}
\def\J{\mathcal{J}}
\def\C{\mathcal{C}}
\def\Q{\mathcal{Q}} 
\def\O{\mathcal{O}}

\def\V{\mathcal{V}}
\def\X{\mathcal{X}}
\def\pn{p_{\!{_{_n}}}\!}
\def\Rpo{\RR^{\pn{\times}1}}

\begin{document}

\title{Two Numerical Approaches for Nonlinear\\ Weakly Singular Integral Equations}

\author{
M. Ahues\footnote{Former professor at Université de Lyon, Saint-Étienne, France. Email: mario.ahues@gmail.com}\;, 
F. Dias d'Almeida\footnote{Centro de Matemática, Universidade do Porto, Portugal. Email: falmeida@fe.up.pt}\;,
R. Fernandes\footnote{Centro de Matemática, Universidade do Minho, Portugal. Email: rosario@math.uminho.pt}\;,  
P.~B. Vasconcelos\footnote{Centro de Matemática, Universidade do Porto, Portugal. Email: pjv@fep.up.pt}
}

\date{}

\maketitle

\abstract{\noindent
Singularity subtraction for linear weakly singular Fredholm integral equations of the second kind is generalized to nonlinear integral equations. Two approaches are presented: The Classical Approach discretizes the nonlinear problem, and uses some finite dimensional linearization process to solve numerically the discrete problem. Its convergence is proved under mild hypotheses on the nonlinearity and the quadrature rule of the singularity subtraction scheme. The New Approach is based on linearization of the problem in its infinite dimensional setting, and discretization of the sequence of linear problems by singularity subtraction. It is more efficient than the former, as two numerical experiments confirm.}

\section{A Brief Introduction to Singularity Subtraction}\label{sec:Intro}

The reference Banach space is the set $\X:=C^0(\lbr a\,,b\rbr,\RR)$ of continuous functions from $\lbr a\,,b\rbr$ to $\RR$, with the supremum norm. We consider the Urysohn integral operator $K$ defined by
\begin{eqnarray*}
K(x)(s)&\!\!:=\!\!&\integ{a}{b}{\!\!\!\!g(|s-t|)N(s,t,x(t))}{t}\mbox{ for all }x\in\X,\ s\in\lbr a\,,b\rbr,
\end{eqnarray*}
where $g$ is a weakly singular function in the following sense:
\begin{enumerate}
\item $g(0^+)=+\infty$
\item $g\in L^1(\lbr0\,,b\!-\!a\rbr\,,\RR)$
\item $g$ is either 
\begin{itemize}
\item[({3a})] a continuous decreasing nonnegative function on $\rbr 0\,,b\!-\!a\rbr$,
\end{itemize}
or
\begin{itemize}
\item[({3b})] a continuous decreasing nonnegative function on $\rbr 0\,,(b\!-\!a)/2\rbr$, symmetric with respect to the midpoint of $\lbr0\,,b\!-\!a\rbr$.
\end{itemize}
\end{enumerate}
A typical example of the case (3a) is
$$
g(r):=\dfrac{1}{2\sqrt{r}}\mbox{ for all }r\in\rbr0\,,1\rbr,
$$
which will be treated numerically in Subsection \ref{subsec:1}, and one of the case (3b) is 
$$
g(r):=\log2-\log(1-\cos 2\pi r)\mbox{ for all }r\in\rbr0\,,1\lbr,
$$
where $\log$ denotes the Neperian logarithm. This example will be handled in Subsection \ref{subsec:2}.

\medskip

\noindent
The factor $N$, containing the values $x(t)\in\RR$ of the functional variable $x\in\X$ for $t\in\lbr a\,,b\rbr$, is a continuous function
\begin{eqnarray*}
N:\lbr a\,,b\rbr \times\lbr a\,,b\rbr \times\RR&\to&\RR\\
(s,t,u)&\mapsto&N(s,t,u)
\end{eqnarray*}
with continuous partial derivative with respect to the third variable.

\medskip

\noindent
Then $K$ maps $\X$ into itself, is compact and  Fréchet-differentiable at any point of $\X$.

\medskip

\noindent
When $N(s,t, x(t)):=\kappa(s,t)\,x(t)$ for some continuous function $\kappa : \lbr a\,,b\rbr {\times}\lbr a\,,b\rbr \to\RR$, then $K$ is a linear bounded operator from $\X$ into itself.

\medskip

\noindent
In this paper, we are interested in the general, possibly nonlinear, case.

\medskip

\noindent
The main idea of the singularity subtraction method is to compensate the singularity of the function $(s,t)\mapsto g(|s-t|)$ along the diagonal $s=t$, by multiplying $g(|s-t|)$ by the factor $N(s,t, x(t))-N(s,s,x(s))$ that tends to $0$ as $t\to s$.

\medskip

\noindent
This leads to rewrite $K$ as
\begin{eqnarray}\label{K-ss}
K(x)(s)&=&\integ{a}{b}{\!\!\!g(|s-t|)(N(s,t,x(t))\!-\!N(s,s,x(s)))}{t}
+N(s,s,x(s))\integ{a}{b}{\!\!\!g(|s\!-\!t|)}{t}.
\end{eqnarray}
The singularity subtraction method builds an approximation of $K$ as it is written in \eqref{K-ss}, and, as described in \cite{An1981} for the linear case, it is a double approximation scheme consisting of {\bf truncation} and {\bf numerical integration}.

\medskip

\noindent
The ideas worked out in \cite{AhLaLi2001,An1981} for the linear case, are extended here to the nonlinear case.

\medskip

\noindent{\bf Truncation:} Given $\delta\in\;\rbr 0\,,b-a\lbr$, we replace $g$ with the so-called $\delta$-truncated approximation $g_{_\delta}$. This function coincides with $g$ outside an interval of length $\delta$, contaning the abscissa that provoques the singularity, and is constantly equal to $g(\delta)$ in that interval. Hence $g_{_\delta}$ is a continuous function.  In the sequence of singularity subtraction approximations, the role of $\delta$ is played by a sequence $(\delta_{_n})_{n\ge2}$ in $\rbr 0\,,b-a\lbr$ leading to the function $g_{_{_{\!\delta_{_n}}}}$ defined by
\begin{eqnarray*}
g_{_{\delta_{_{n}}}}\!\!(r)&\!\!:=\!\!&
\left\{
\begin{array}{ll}
g(\delta_{_n})&\mbox{ for all }r\in\lbr0\,,\delta_{_n}\rbr,\\
g(r)&\mbox{ for all }r\in\,\rbr \delta_{_n}\,,b-a\rbr 
\end{array}
\right.
\end{eqnarray*}
for a function $g$ satisfying (3a), or
\begin{eqnarray*}
g_{_{\delta_{_{n}}}}(r)&\!\!:=\!\!&
\left\{
\begin{array}{ll}
g(\delta_{_{n}})&\mbox{ for all }r\in\lbr0\,,\delta_{_{n}}\rbr,\\
g(r)&\mbox{ for all }r\in\rbr\delta_{_{n}}\,,b\!-\!a\!-\!\delta_{_{n}}\lbr,\\
g(\delta_{_{n}})&\mbox{ for all }r\in\lbr b\!-\!a\!-\!\delta_{_{n}}\,,b\!-\!a\rbr
\end{array}\right.
\end{eqnarray*}
for a function $g$ satisfying (3b).

\medskip

\noindent{\bf Numerical integration:} To proceed with the singularity subtraction scheme ---like in the linear case--- we define a general grid with $n\ge2$ points on $\lbr a\,,b\rbr$:
\begin{eqnarray}\label{grid}
a\le \widehat{t}_{_{n,1}}<\widehat{t}_{_{n,2}}<\ldots<\widehat{t}_{_{n,n}}\le b.
\end{eqnarray}
This grid is called the basic grid, and it determines $n-1$ subintervals of $\lbr a\,,b\rbr$.

\medskip

\noindent
The first integral of \eqref{K-ss}, after replacing $g$ with $g_{_{_{\!\delta_{_n}}}}$, is approximated by some quadrature rule $\Q_{_{n}\!}$ with $\pn$ nodes depending on the nodes of the basic grid. For instance, if $\Q_{_{n}\!}$ is the composite trapezoidal rule, then the quadrature grid is the basic grid, so $\pn=n$; if $\Q_{_{n}\!}$ is the composite Simpson rule, then its nodes are the points of the basic grid and the mid-points of the corresponding subintervals, and hence $\pn=2n-1$. For some other rules $\Q_{_{n}\!}$, the nodes are the so-called Gaussian points that are obtained by shifting to each subinterval of the basic grid the zeros of a polynomial of a given degree $m$ belonging to a complete sequence of orthogonal polynomials in some particular Hilbert space, and hence $\pn=m(n-1)$. In this paper, numerical experiments will be worked out with the midpoint rectangular rule for which $\pn=n-1$.

\medskip

\noindent
As in \cite{AhLaLi2001}, intervals of integer numbers are denoted by $\ilbr\,\cdot\,,\,\cdot\irbr$. 

\medskip

\noindent
Let the $\pn$ nodes of $\Q_{_{n}\!}$ be denoted by $t_{\pn,j}$, $j\in\ilbr 1\,,\,\pn\irbr$, and numbered so that
$$
a \le t_{\pn,1} < \cdots < t_{\pn,\pn} \le b.
$$
Let the $\pn$ weights of $\Q_{_{n}\!}$ be denoted by $w_{\pn,j}$, $j\in\ilbr 1\,,\,\pn\irbr$. We suppose that they are all positive, and that there exists a constant $\hat{\gamma}>0$ satisfying
\begin{equation}
\label{H}
\sum\limits_{t_{\pn,j}\in\J}\!\!\!\!\!w_{\pn,j}\le\hat{\gamma}\,(d-c)\;\mbox{ when }\;a \le c < d \le b,\;\mbox{ and }\;\J\;\mbox{ is }\; \rbr c\,,d\rbr \;\mbox{ or }\:\lbr c\,,d\lbr 
\end{equation}
(cf. hypothesis (H) in \cite{AhLaLi2001}, page 225). Almost all commonly used quadrature rules satisfy \eqref{H}. The constant $\hat{\gamma}$ plays an active role in the proof of Theorem \ref{thm:classical}.

\medskip

\noindent
Ideally, $\integ{a}{b}{\!\!\!g(|s-t|)}{t}$ should be available in closed form, and this is sometimes possible. For instance, if a primitive $G$ of $g$ is available, then
\begin{eqnarray}\label{G}
f(s):=\integ{a}{b}{\!\!\!g(|s-t|)}{t}=G(s\!-\!a)+G(b\!-\!s)-2\,G(0)\mbox{ for all }s\in\lbr a\,,b\rbr.
\end{eqnarray}
Otherwise, a specially fine numerical quadrature formula should give an accurate value of this integral for any fixed value of $s\in\lbr a\,,b\rbr$.

\medskip

\noindent
Formula \eqref{G} is particularly useful to prove some properties of $f$ such as its symmetry with respect to $(a+b)/2$, and that it is a strictly increasing function on $\lbr a\,,(a+b)/2\rbr$ for example.

\medskip

\noindent
The exact problem to be solved numerically is the following: 
\begin{eqnarray}\label{Equation}
\mbox{For $y\in\X$, find $\varphi\in\X$ such that }\varphi=K(\varphi)+y,
\end{eqnarray}
i.e. find $\varphi\in\X$ such that
$$
\F(\varphi)=0,
$$
where $\F:\X\to\X$ is the operator defined by 
$$
\F(x):=x-K(x)-y\mbox{ for all }x\in\X.
$$
We assume that the linear bounded operator $I-K^{\,\prime}(\varphi):\X\to\X$ has a bounded inverse $(I-K^{\,\prime}(\varphi))^{-1}:\X\to\X$, where $K^{\,\prime}(x)$ denotes the Fréchet-derivative of $K$ at $x\in\X$. Hence $\varphi$ is an isolated solution of \eqref{Equation}.

\medskip

\noindent
Two approaches, both using a singularity subtraction scheme, are presented in this paper. The first one, called here the Classical Approach, has been proposed by the authors in \cite{AhDAFeVa2019}. The second one is presented here for the first time. It constitutes an extension to singularity subtraction approximations of the method developed by the authors for norm convergent projection approximations in \cite{GrAhDA2014,GrVaAh2016}.

\section{Basics on Convergence}\label{sec:basics}

Consider the basic grid \eqref{grid}, define $h_{n,j}:=\widehat{t}_{_{n,j+1}}-\widehat{t}_{_{n,j}}$ for $j\in\ilbr 1\,,\,n-1\irbr$, and 
$$
h_n:=\max\limits_{j\in\ilbr 1\,,\,n-1\irbr}h_{n,j}.
$$
The singularity subtraction technique, as presented in \cite{An1981}, relates truncation and numerical integration through the following condition on the sequences $(\delta_{_n})_{n\ge2}$ and $(\Q_{_{n}\!})_{n\ge 2}$: 

\medskip

\noindent
There exist constants $\alpha_1>0$ and $\beta_1>0$ such that

\medskip

\centerline{$\alpha_1 h_n\le \delta_{_n}\le\beta_1 h_n\;\mbox{ for all }\;n\ge2,$}

\medskip

\noindent
i.e. the width of truncation must tend to zero at the same rate as the mesh size.

\medskip

\noindent
These considerations lead to approximate $K$, as written in \eqref{K-ss}, by the following operator $K_n$: For all $x\in\X$, and $s\in\lbr a\,,b\rbr$,
\begin{eqnarray*}
K_n(x)(s)&\!\!:=\!\!&\sum\limits_{j=1}^{\pn}\!w_{\pn,j}g_{_{\delta_{_{n}}}}\!\!(|s\!-\!t_{\pn,j}|)\big(N(s,t_{\pn,j},x(t_{\pn,j}))-N(s,s,x(s))\big)\\ 
&&+N(s,s,x(s))\integ{a}{b}{\!\!\!g(|s-t|)}{t}.
\end{eqnarray*}

\medskip

\noindent
The approximate equation, to be solved exactly, is: Find $\varphi_{_{\!\!_n}\!}\in\X$ such that 
\begin{eqnarray}\label{Approx-newapp}
\varphi_{_{\!\!_n}\!}&=&K_n(\varphi_{_{\!\!_n}\!})+y,
\end{eqnarray}
i.e. $\F_n(\varphi_{_{\!\!_n}\!})=0$, where $\F_n:\X\to\X$ is the operator defined by 
$$
\F_n(x):=x-K_n(x)-y\mbox{ for all }x\in\X.
$$
Let $\stackrel{\mathrm{p}}\to$ denote pointwise convergence, $\stackrel{\mathrm{n}}\to$ norm convergence, $\stackrel{\mathrm{cc}}\to$ collectively compact convergence (cf. \cite{An1971}), and $\stackrel{\nu}\to$ the $\nu$-convergence (cf. \cite{AhLaLi2001}).

\medskip

\noindent
The Fréchet-derivatives $T:=K^{\,\prime}$ and $T_n:=(K_n)^{\,\prime}$ at $\varphi$ are given by:
\begin{eqnarray*}
\nonumber
(T(\varphi) f)(s)&=& \integ{a}{b}{\!\!\!g(|s-t|)\dfrac{\partial N}{\partial u}(s,t,\varphi(t))f(t)}{t},\\
(T_n(\varphi)f)(s)&=& \sum\limits_{j=1}^{\pn}\!w_{\pn,j}\,g_{_{\delta_{_{n}}}}\!\!(|s\!-\!t_{\pn,j}|)\dfrac{\partial N}{\partial u}(s,t_{\pn,j},\varphi (t_{\pn,j}))f(t_{\pn,j})\\
\nonumber&& -\sum\limits_{j=1}^{\pn}\!w_{\pn,j}\,g_{_{\delta_{_{n}}}}\!\!(|s\!-\!t_{\pn,j}|)\dfrac{\partial N}{\partial u}(s,s,\varphi(s))f(s)\\
\nonumber&&
+\dfrac{\partial N}{\partial u}(s,s,\varphi(s))f(s) \integ{a}{b}{\!\!\!g(|s-t|)}{t}
\end{eqnarray*}
for $f\in\X$, $s\in \lbr a\,,b\rbr$.

\medskip

\noindent
We define, for all $x\in\X$, and $s\in\lbr a\,,b\rbr$,
\begin{eqnarray*}
(U x)(s)&\!\!:=\!\!&\integ{a}{b}{\!\!\!\!g(|s-t|)x(t)}{t},\\\\
(U_n x)(s)&\!\!:=\!\!&\sum\limits_{j=1}^{\pn}\!w_{\pn,j} g_{_{\delta_{_{n}}}}\!\!(|s\!-\!t_{\pn,j}|) \, x(t_{\pn,j}).
\end{eqnarray*}
By \eqref{H}, $U_n\stackrel{\mathrm{cc}}\to U$, so $U_n \stackrel{\mathrm{p}}\to U$ (cf. Proposition 4.18 in \cite{AhLaLi2001}, page 227).

\begin{lem}\label{lem:convs}
$T_n(\varphi) \stackrel{\mathrm{p}}\to T(\varphi)$ and $T_n(\varphi) \stackrel{\nu}\to T(\varphi)$.
\end{lem}
\begin{proof}
As $T_n(\varphi)$ and $T(\varphi)$ are bounded linear operators, we use the results of \cite{An1981}.

\medskip

\noindent
Let us consider the decomposition $T_n(\varphi)=S_n(\varphi)+E_n(\varphi)$, where
\begin{eqnarray*}
(S_n(\varphi)f)(s)&\!\!:=\!\!&\sum\limits_{j=1}^{\pn}\!w_{\pn,j}\,g_{_{\delta_{_{n}}}}\!\!(|s\!-\!t_{\pn,j}|)\dfrac{\partial N}{\partial u}(s,t_{\pn,j},\varphi(t_{\pn,j}))f(t_{\pn,j}),\\
(E_n(\varphi)f)(s)&\!\!:=\!\!&\dfrac{\partial N}{\partial u}(s,s,\varphi(s))f(s)\Big(\integ{a}{b}{\!\!\!g(|s-t|)}{t}-\sum\limits_{j=1}^{\pn}\!w_{\pn,j}\,g_{_{\delta_{_{n}}}}\!\!(|s\!-\!t_{\pn,j}|)\Big)\\
&\!\!\,=\!\!&\dfrac{\partial N}{\partial u}(s,s,\varphi(s))f(s)\big((U e) (s) -(U_n e)(s)\big)
\end{eqnarray*}
for all $f\in\X$, and $e(s):=1$ for all $s \in \lbr a\,,b\rbr$.

\medskip

\noindent
Since $(s,t,u)\mapsto\dfrac{\partial N}{\partial u}(s,t,u)$ is a continuous function, and since \eqref{H} holds, then $T_n(\varphi)$ and $S_n(\varphi)$ satisfy the hypotheses of Proposition 4.18 in \cite{AhLaLi2001}, page 227, and $S_n(\varphi) \stackrel{\mathrm{cc}}\to T(\varphi)$. 

\medskip

\noindent
Hence $S_n(\varphi) \stackrel{\mathrm{p}}\to T(\varphi)$. Since $T(\varphi)$ is compact because $K$ is compact. $S_n(\varphi) \stackrel{\nu}\to T(\varphi)$.

\medskip

\noindent
For all $f\in\X$ such that $\|f\|=1$,
$$
\|E_n (\varphi)f\|\le\|U_ne-Ue\|\max\limits_{s\in\lbr a\,,b\rbr}\Big|\dfrac{\partial N}{\partial u}(s,s,\varphi(s))\Big|
$$
that tends to $0$ as $n\to\infty$ because $U_n \stackrel{\mathrm{p}}\to U$. Hence $E_n (\varphi)\stackrel{\mathrm{n}}\to O$, so $T_n(\varphi) \stackrel{\mathrm{p}}\to T(\varphi)$.

\medskip

\noindent
This proves that $T_n(\varphi) \stackrel{\nu}\to T(\varphi)$ (cf. Lemma 2.2 (b) (i) in \cite{AhLaLi2001}, page 73).	
\end{proof}

\begin{lem}\label{lem:bdn}
For all $n$ large enough, $I-T_n(x)$ is invertible for all $x$ close enough to $\varphi$, and the inverse operator $(I-T_n(x))^{-1}$ is uniformly bounded with respect to $n$.
\end{lem}
\begin{proof}
Since $(I-T(\varphi))^{-1}$ exists, and $T_n(\varphi)\stackrel{\nu}\to T(\varphi)$, there exists $n_0\ge2$ such that, for all $n\ge n_0 $,
$$
\|(I-T(\varphi))^{-1}\|\;\| \big(T_n(\varphi) - T(\varphi) \big)T_n (\varphi) \| < 1.
$$
Hence $(I-T_n(\varphi))^{-1}$ exists and  is uniformly bounded (cf. \cite{An1981}, page 413). By continuity, the same holds for $(I-T_n(x))^{-1}$ for all $x$ close enough to $\varphi$.
\end{proof}

\begin{lem}\label{lem-loc-inv}
$\F_n$ is locally invertible with continuous inverse in a neighborhood of $0$.
\end{lem}
\begin{proof}
$I-K_n$ is a continuously differentiable operator from the Banach space $\X$ into itself. By the Inverse Function Theorem, $I-T_n(\varphi)$, being invertible, there is a  neighborhood of $\varphi$  where $I-K_n$ is invertible with continuous inverse in some neighborhood of $y$. Hence $\F_n^{-1}$ exists and is continuous in some neighborhood of $0$.	
\end{proof}

\begin{lem}\label{lem:Fn-conv}
${\big(K_n\big)}_{n \ge 2}$ is pointwise convergent to $K$, and ${\big(\F_n\big)}_{n \ge 2}$ is pointwise convergent to $\F$.
\end{lem}
\begin{proof}
An auxiliary operator $\widehat{K}_n$ is used in the proof. For all $x\in\X$, and $s\in\lbr a\,,b\rbr$, define
\begin{eqnarray*}
\widehat{K}_n(x)(s)&\!\!:=\!\!& \sum\limits_{j=1}^{\pn}\!w_{\pn,j}\,g_{_{\delta_{_{n}}}}\!\!(|s\!-\!t_{\pn,j}|)N(s,t_{\pn,j},x(t_{\pn,j})).
\end{eqnarray*}
$K_n$ can be rewritten as
\begin{eqnarray*}
K_n(x)(s)&=&\widehat{K}_n(x)(s)+N(s,s,x(s))(U-U_n)e(s).
\end{eqnarray*}
Define 
$$
\sigma(x):=\max\limits_{(s,t)\in\lbr a\,,b\rbr{\times}\lbr a\,,b\rbr}|N(s,t,x(t))|,
$$
which is finite because of the continuity of $N$ in its three variables, and that of $x$ in its single one. In the linear case, $\sigma(x)=\rho\|x\|$ for some constant $\rho>0$. Now,
$$
|K_n(x)(s)-\widehat{K}_n(x)(s)|\le \sigma(x)\|(U-U_n)e\|\to0\mbox{ and }\|K_n(x)-\widehat{K}_n(x)\|\to0\mbox{ as }n\to\infty,
$$
since $U_n\stackrel{\rm p}\to U$. Following the ideas of the proof of Proposition 4.18 in \cite{AhLaLi2001}, we decompose
$$
\widehat{K}_n(x)-K(x)=\lambda_\delta+\mu_n+\eta_n,
$$
where $\lambda_\delta$, $\mu_n$ and $\eta_n$ are defined as follows. Let $\hat{\gamma}>0$ be the constant introduced in \eqref{H}. Given $\epsilon>0$, there exists $\delta\in\,\rbr 0\,,b-a\rbr$ such that $\integ{0}{\delta}{\!\!g(u)}{u}<\dfrac{\epsilon}{18}\min\{1\,,\,\dfrac{1}{3\hat{\gamma}}\}$. Set 
\begin{eqnarray*}
\lambda_\delta(s)&\!\!:=\!\!&\!\int\limits_{\max\{a,s-\delta\}}^{\min\{b,s+\delta\}}{\!\!\!\!\!\!\!\!\!(g(\delta)-g(|s-t|))N(s,t,x(t))}{dt},\\
\mu_n(s)\!&\!\!:=\!\!&\!\sum\limits_{j=1}^{\pn}\!w_{\pn,j}(g_{_{\delta_{_{n}}}}\!\!(|s\!-\!t_{\pn,j}|)-g_{_\delta}(|s\!-\!t_{\pn,j}|))N(s,t_{\pn,j},x(t_{\pn,j})),\\
\eta_n(s)&\!\!:=\!\!&\!\sum\limits_{j=1}^{\pn}\!w_{\pn,j}g_{_\delta}(|s\!-\!t_{\pn,j}|)N(s,t_{\pn,j},x(t_{\pn,j}))\!-\!\!\!\integ{a}{b}{\!\!\!\!g_{_\delta}(|s\!-\!t|)N(s,t,x(t))}{t}.
\end{eqnarray*}
Then the following upper bounds hold for all $n$ greater than some integer $n_0(x)$:
\begin{eqnarray*}
&&|\lambda_\delta(s)|\le6\;\sigma(x)\integ{0}{\delta}{\!\!\!g(u)}{u}\,\le\,\dfrac{\sigma(x)}{3}\epsilon,\\
&&|\mu_n(s)|\le\sigma(x)\big(\!\!\!\!\!\!\sum\limits_{|s-t_{\pn,j}|<\delta}\!\!\!\!\!\!w_{\pn,j}g_{_{\delta_{_{n}}}}\!\!(|s\!-\!t_{\pn,j}|)+2\hat{\gamma}\delta g(\delta)\big)\,\le\,\dfrac{\sigma(x)}{3}\epsilon,\\
&&|\eta_n(s)|\le\dfrac{\sigma(x)}{3}\epsilon.
\end{eqnarray*}
Since
\begin{eqnarray*}
|\widehat{K}_n(x)(s)-K(x)(s)|&\le&|\lambda_\delta(s)|+|\mu_n(s)|+|\eta_n(s)|\le\sigma(x)\epsilon,
\end{eqnarray*}
we conclude that $\widehat{K}_n\stackrel{\rm p}\to K$, $K_n\stackrel{\rm p}\to K$, and $\F_n\stackrel{\rm p}\to\F$.
\end{proof}

\section{The Classical Approach: Discretize First}\label{sec:ss}

We recall that $\big(\Q_{_{n}\!}\big)_{n\ge2}$ is a sequence of composite quadrature rules with nodes $\big(t_{\pn,j}\big)_{j=1}^{\pn}$ and weights $\big(w_{\pn,j}\big)_{j=1}^{\pn}$, satisfying \eqref{H}. 

\medskip

\noindent
If we take the values of \eqref{Approx-newapp} at $t_{\pn,i}$, $i\in\ilbr 1\,,\,\pn\irbr$, we get the following, possibly nonlinear, system of order $\pn$:
\begin{eqnarray*}
\msf{x}_n(i)&\!=\!\!&\sum\limits_{j=1}^{\pn}\!\!w_{\pn,j}\,g_{_{\delta_{_{n}}}}\!\!(|t_{\pn,i}\!-\!t_{\pn,j}|)( N(t_{\pn,i},t_{\pn,j},\msf{x}_n(j))\!-\!N(t_{\pn,i},t_{\pn,i},\msf{x}_n(i)))\\
&&+\,N(t_{\pn,i},t_{\pn,i},\msf{x}_n(i))\integ{a}{b}{\!\!\!g(|t_{\pn,i}-t|)}{t}+y(t_{\pn,i}),
\end{eqnarray*}
where the unknowns correspond to the grid values of  $\varphi_{_{\!\!_n}\!}$:
$$
\msf{x}_n(i):=\varphi_{_{\!\!_n}\!}(t_{\pn,i})\mbox{ for all }i\in\ilbr 1\,,\,\pn\irbr.
$$
This system of order $\pn$ can be written as
\begin{eqnarray}\label{classical-LS}
\msf{F}_n(\msf{x}_n)&=&0,
\end{eqnarray}
where, for all $\msf{x}\in\Rpo$, and $i\in\ilbr 1\,,\,\pn\irbr$,
\begin{eqnarray*}
\msf{F}_n(\msf{x})(i)&\!\!:=\!\!&\msf{x}(i)-\sum\limits_{j=1}^{\pn}\!w_{\pn,j}\,g_{_{\delta_{_{n}}}}\!\!(|t_{\pn,i}\!-\!t_{\pn,j}|)((N(t_{\pn,i},t_{\pn,j},\msf{x}(j))-N(t_{\pn,i},t_{\pn,i},\msf{x}(i)))\\
&&-N(t_{\pn,i},t_{\pn,i},\msf{x}(i))\integ{a}{b}{\!\!\!g(|t_{\pn,i}-t|)}{t}-y(t_{\pn,i})\\
&=&\msf{x}(i)-\sum\limits_{j=1}^{\pn}\!w_{\pn,j}\,g_{_{\delta_{_{n}}}}\!\!(|t_{\pn,i}\!-\!t_{\pn,j}|)N(t_{\pn,i},t_{\pn,j},\msf{x}(j))\\
&&+N(t_{\pn,i},t_{\pn,i},\msf{x}(i))\Big(\sum\limits_{j=1}^{\pn}\!w_{\pn,j}\,g_{_{\delta_{_{n}}}}\!\!(|t_{\pn,i}\!-\!t_{\pn,j}|)-\!\!\integ{a}{b}{\!\!\!g(|t_{\pn,i}-t|)}{t}\Big)-y(t_{\pn,i}).
\end{eqnarray*}
System \eqref{classical-LS} must be solved accurately by some numerical method like, for instance, Gauss' method in the linear case, and Newton's method ---as described in the sequel--- in the nonlinear case.

\medskip

\noindent
The Jacobian matrix of $\msf{F}_n:\Rpo\to\Rpo$ at $\msf{x}\in\Rpo$ is given by
\begin{eqnarray*}
\msf{F}_n^{\,\prime}(\msf{x})(i,j)&=&\delta_{i,j}-w_{\pn,j}\,g_{_{\delta_{_{n}}}}\!\!(|t_{\pn,i}\!-\!t_{\pn,j}|)\dfrac{\partial N}{\partial u}(t_{\pn,i},t_{\pn,j},\msf{x}(j))\\
&&+\delta_{i,j}\dfrac{\partial N}{\partial u}(t_{\pn,i},t_{\pn,i},\msf{x}(i))\Big(\sum\limits_{\ell=1}^{\pn} w_{\pn,j}\,g_{_{\delta_{_{n}}}}\!\!(|t_{\pn,i}\!-\! t_{\pn,\ell}|)-\integ{a}{b}{\!\!\!g(|t_{\pn,i}-t|)}{t}\Big),
\end{eqnarray*}
where $\delta_{i,j}$ is the Kronecker delta, and $i,j\in\ilbr 1\,,\,\pn\irbr$.

\medskip

\noindent
The Newton's sequence $\big(\msf{x}_n^{[k]}\big)_{k\ge0}$ in $\Rpo$ is defined, for a given starting column $\msf{x}_n^{[0]}\in\Rpo$, by
\begin{eqnarray*}
\msf{F}_n^{\,\prime}(\msf{x}_n^{[k]})\,\msf{x}_n^{[k+1]}&=&\msf{F}_n^{\,\prime}(\msf{x}_n^{[k]})\,\msf{x}_n^{[k]}-\msf{F}_n(\msf{x}_n^{[k]})\mbox{ for all }k\ge0,
\end{eqnarray*}
where $\msf{x}_n^{[k+1]}$ is the unknown, i.e.
\begin{eqnarray*}
(\msf{I}_n-\msf{A}_n^{[k]}-\msf{B}_n^{[k]})\,\msf{x}_n^{[k+1]}&=&\msf{a}_n^{[k]},
\end{eqnarray*}
where $\msf{I}_n$ is the identity matrix of order $\pn$, and, for all $i,j\in\ilbr 1\,,\,\pn\irbr$,
\begin{eqnarray}\label{Ank-classical}
\msf{A}_n^{[k]}(i,j)&\!\!\!:=\!\!\!&w_{\pn,j}\,g_{_{\delta_{_{n}}}}\!\!(|t_{\pn,i}\!-\!t_{\pn,j}|)\dfrac{\partial N}{\partial u}(t_{\pn,i},t_{\pn,j},\msf{x}_n^{[k]}(j)),\\
\nonumber\\
\label{Dnk-classical}
\msf{B}_n^{[k]}(i,j)&\!\!\!:=\!\!\!&\delta_{i,j}\dfrac{\partial N}{\partial u}(t_{\pn,i},t_{\pn,i},\msf{x}_n^{[k]}(i))\Big(\integ{a}{b}{\!\!\!g(|t_{\pn,i}\!-\!t|)}{t}\!-\!\sum\limits_{\ell=1}^{\pn} w_{\pn,\ell}\,g_{_{\delta_{_{n}}}}\!\!(|t_{\pn,i}\!-\!t_{\pn,\ell}|)\Big),
\end{eqnarray}
and
\begin{eqnarray}\label{bnk-classical}
\msf{a}_n^{[k]}(i)&:=&y(t_{\pn,i})
\end{eqnarray}
\begin{eqnarray*}
\hspace*{-0.75 cm}
&&\!+\Big(N(t_{\pn,i},t_{\pn,i},\msf{x}_n^{[k]}(i))\!-\!\msf{x}_n^{[k]}(i)\dfrac{\partial N}{\partial u}(t_{\pn,i},t_{\pn,i},\msf{x}_n^{[k]}(i))\Big)\!\Big(\integ{a}{b}{\!\!\!g(|t_{\pn,i}\!-\!t|)}{t}\!-\!\!\sum\limits_{\ell=1}^{\pn}\!w_{\pn,\ell}g_{_{\delta_{_{n}}}}\!\!(|t_{\pn,i}\!-\!t_{\pn,\ell}|)\Big)\\
\hspace*{-0.75 cm} &&\!+\sum\limits_{j=1}^{\pn}\!w_{\pn,j}\,g_{_{\delta_{_{n}}}}\!\!(|t_{\pn,i}\!-\!t_{\pn,j}|)\Big(N(t_{\pn,i},t_{\pn,j},\msf{x}_n^{[k]}(j))\!-\!\dfrac{\partial N}{\partial u}(t_{\pn,i},t_{\pn,j},\msf{x}_n^{[k]}(j))\msf{x}_n^{[k]}(j)\Big).
\end{eqnarray*}
For $n$ fixed, and under suitable hypotheses on $\msf{F}_n$ and $\msf{x}_n^{[0]}$, the sequence $\big(\msf{x}_n^{[k]}\big)_{k\ge0}$ is quadratically convergent with limit $\msf{x}_n$, the column of values of $\varphi_{_{\!\!_n}\!}$ at the nodes $\big(t_{\pn,j}\big)_{j=1}^{\pn}$.

\begin{thm}[On the convergence of the Classical Approach]\quad{}\newline
\label{thm:classical}
The sequence $\big(\varphi_{_{\!\!_n}\!}\big)_{n\ge2}$ is convergent with limit $\varphi$.
\end{thm}

\begin{proof}
Since $\F$ and $\F_n$ are invertible and Fréchet-differentiable, the derivative of their inverses at $0$ is equal to the inverse of the derivative of the direct operators at the inverse image of $0$, and the integral form of the Mean Value Theorem for Derivatives gives:
\begin{eqnarray*}
\varphi_{_{\!\!_n}\!}-\varphi&=&\F_n^{-1}(0)-\F^{-1}(0)=\F_n^{-1}(\F(\varphi))-\F_n^{-1}(\F_n(\varphi))\\
&=&\integ{0}{1}{(\F_n^{-1})^{\,\prime}(\F_n(\varphi)+t(\F(\varphi)-\F_n(\varphi))}{t}\:(\F(\varphi)-\F_n(\varphi)).
\end{eqnarray*}
Also,
\begin{eqnarray*}
&&\F_n(\varphi_{_{\!\!_n}\!})-\F_n(\varphi)=\integ{0}{1}{\F_n^{\,\prime}(\varphi+t(\varphi_{_{\!\!_n}\!}-\varphi))}{t}\;(\varphi_{_{\!\!_n}\!}-\varphi).
\end{eqnarray*}
Since the sequence $\big(\F_n\big)_{n\ge2}$ is pointwise convergent to $\F$ and $\F(\varphi)=0$, then
$$
v_n(t):=\F_n(\varphi)+t(\F(\varphi)-\F_n(\varphi))
$$
tends to $0$ uniformly in $t\in\lbr0\,,1\rbr$ as $n\to\infty$. On the other hand, 
$$
(\F_n^{-1})^{\,\prime}(v_n(t))=(I-T_n(u_n(t)))^{-1}
$$
is uniformly bounded for all $n$ large enough and $t\in\lbr0\,,1\rbr$, where 
$$
u_n(t):=\F_n^{-1}(v_n(t)).
$$
Also, $\F_n^{\,\prime}(x)=I-T_n(x)$ is bounded uniformly in $x$ for all $x$ in any bounded set of $\X$. Since $\F_n(\varphi_{_{\!\!_n}\!})=0$, there exist $n$-independent constants $\alpha_2>0$ and $\beta_2>0$ such that 
$$
\alpha_2\|\F_n(\varphi)\|\le\|\varphi_{_{\!\!_n}\!}-\varphi\|\le \beta_2\|\F_n(\varphi)\|.
$$
But $\F_n(\varphi)=K(\varphi)-K_n(\varphi)$, so
\begin{eqnarray*}
\alpha_2\|K_n(\varphi)-K(\varphi)\|\le\|\varphi_{_{\!\!_n}\!}-\varphi\|\le \beta_2\|K_n(\varphi)-K(\varphi)\|.
\end{eqnarray*}
This proves that the sequence $(\varphi_{_{\!\!_n}\!})_{n\ge2}$ is convergent with limit $\varphi$. 
\end{proof}
\noindent
The previous nested bound shows that the rate of convergence of $\big(\varphi_{_{\!\!_n}\!}\big)_{n\ge2}$ to $\varphi$ is the same as the rate of convergence of $\big(K_n(\varphi)\big)_{n\ge2}$ to $K(\varphi)$. In other words, the quality of the approximate solution $\varphi_{_{\!\!_n}\!}$ and the quality of the approximate operator $K_n$ at the exact solution are of the same order.

\section{A New Approach: Linearize First}\label{sec:newapp}

We can tackle the numerical resolution of the nonlinear problem following a New Approach: first linearize the problem in the infinite dimensional space $\X$ with the Newton-Kantorovich method, and then solve numerically at each iteration, the linear problem issued from this method using a discretization scheme.

\medskip

\noindent
The Newton-Kantorovich step number $k$ applied to produce a sequence $\big(\varphi^{[k]}\big)_{k\ge0}$ having as its limit a function $\varphi$ such that $\F(\varphi)=0$ can be written as the linear problem with unknown $\varphi^{[k+1]}$
$$
\F^{\,\prime}(\varphi^{[k]})\varphi^{[k+1]}=\F^{\,\prime}(\varphi^{[k]})\varphi^{[k]}-\F(\varphi^{[k]}),
$$
where $\varphi^{[0]}$ must be properly chosen by the user. Remark that $\varphi^{[k+1]}$ is given by
\begin{eqnarray}\label{NK}
\varphi^{[k+1]}=T(\varphi^{[k]})\varphi^{[k+1]}+z^{[k]}\mbox{ for all }k\ge0,
\end{eqnarray}
where
\begin{eqnarray*}
z^{[k]}:=K(\varphi^{[k]})-T(\varphi^{[k]})\varphi^{[k]}+y\mbox{ for all }k\ge0.
\end{eqnarray*}
In the sequel, $\O(\varphi\,,r)$ denotes the open ball of center $\varphi$ and radius $r$, and $\C(\varphi\,,r)$ denotes the closed ball of center $\varphi$ and radius $r$. 

\medskip

\noindent
Our approach leads to solve equation \eqref{NK} numerically at each step $k$ of the Newton-Kantorovich process. We shall thus build a sequence $\big(\varphi_{\!\!_n}^{[k]}\big)_{n\ge2}$ such that $\varphi_{\!\!_n}^{[k]}$ is a sufficiently good approximation to $\varphi^{[k]}$ for all $n$ large enough but fixed, and we expect that, for a such value of $n$,
$$
\lim\limits_{k\to+\infty}\varphi_{\!\!_n}^{[k]}=\varphi.
$$
$T(\varphi^{[k]}):\X\to\X$ is a weakly singular linear Fredholm integral operator, and it can be approximated with the singularity subtraction scheme involving the linear bounded operator $T_n(\varphi_{\!\!_n}^{[k]})$ given by:

\medskip

\noindent
For all $f\in\X$ and $s\in\lbr a\,,b\rbr$,
\begin{eqnarray}\label{Tnk}
(T_n(\varphi_{\!\!_n}^{[k]})f)(s)\!\!&\!\!\!=\!\!\!&\!\!\sum\limits_{j=1}^{\pn}\!w_{\pn,j}g_{_{\delta_{_{n}}}}\!\!(|s\!-\!t_{\pn,j}|)\dfrac{\partial N}{\partial u}(s,t_{\pn,j},\varphi_{\!\!_n}^{[k]}(t_{\pn,j}))(f(t_{\pn,j})\!-\!f(s))\\ 
&&\nonumber 
+f(s)\!\integ{a}{b}{\!\!\!\!g(|s\!-\!t|)\dfrac{\partial N}{\partial u}(s,t,\varphi_{\!\!_n}^{[k]}(t))}{t}.
\end{eqnarray}
Equation \eqref{NK} is replaced with the approximate equation
\begin{eqnarray}\label{Approx-NK-Eq}
\varphi_{\!\!_n}^{[k+1]}&=&T_n(\varphi_{\!\!_n}^{[k]})\varphi_{\!\!_n}^{[k+1]}+z_n^{[k]},
\end{eqnarray}
where
\begin{eqnarray*}
z_n^{[k]}:=K(\varphi_{\!\!_n}^{[k]})-T_n(\varphi_{\!\!_n}^{[k]})\varphi_{\!\!_n}^{[k]}+y.
\end{eqnarray*}
Evaluating \eqref{Approx-NK-Eq} at each node $t_{\pn,i}$, $i\in\ilbr1\,,\,\pn\irbr$, and defining 
$$
\msf{w}_n^{[k]}(i):=\varphi_{\!\!_n}^{[k]}(t_{\pn,i})\mbox{ for all } i\in\ilbr1\,,\,\pn\irbr,
$$
we get the system of linear equations of order $\pn$
\begin{eqnarray*}
(\msf{I}_n-\msf{C}_n^{[k]}-\msf{D}_n^{[k]})\msf{w}_n^{[k+1]}
&=&\msf{b}_n^{[k]},
\end{eqnarray*}
where $\msf{I}_n$ is the identity matrix of order $\pn$, and \eqref{Tnk} gives for all $i,j\in\ilbr1\,,\,\pn\irbr$,  
\begin{eqnarray}\label{Ank-newapp}
\msf{C}_n^{[k]}(i,j)\!&\!\!\!:=\!\!\!&w_{\pn,j}g_{_{\delta_{_{n}}}}\!\!(|t_{\pn,i}\!-\!t_{\pn,j}|)\dfrac{\partial N}{\partial u}(t_{\pn,i},t_{\pn,j},\varphi_{\!\!_n}^{[k]}(t_{\pn,j})),\\
\nonumber\\
\label{Dnk-newapp}
\msf{D}_n^{[k]}(i,j)\!&\!\!\!:=\!\!\!&\delta_{i,j}\Big(\integ{a}{b}{\!\!\!\!g(|t_{\pn,i}\!-\!t|)\dfrac{\partial N}{\partial u}(t_{\pn,i},t,\varphi_{\!\!_n}^{[k]}(t))}{t}-\sum\limits_{\ell=1}^{\pn}\msf{C}_n^{[k]}(i,\ell)\Big),\\
\nonumber\\
\label{bnk-newapp}
\msf{b}_n^{[k]}(i)\!&\!\!\!:=\!\!\!&K(\varphi_{\!\!_n}^{[k]})(t_{\pn,i})-(T_n(\varphi_{\!\!_n}^{[k]})\varphi_{\!\!_n}^{[k]})(t_{\pn,i})+y(t_{\pn,i})\\
\nonumber\\
\nonumber\!&\!\!\!\,\,=\!\!\!&y(t_{\pn,i})\!+\!\!\integ{a}{b}{\!\!\!g(|t_{\pn,i}\!-\!t|)\big( N(t_{\pn,i},t,\varphi_{\!\!_n}^{[k]}(t))\!-\!\varphi_{\!\!_n}^{[k]}(t_{\pn,i})\dfrac{\partial N}{\partial u}(t_{\pn,i},t,\varphi_{\!\!_n}^{[k]}(t))\big)}{t}\\
\nonumber\\
\nonumber
&&+\!\sum\limits_{j=1}^{\pn}\!w_{\pn,j}g_{_{\delta_{_{n}}}}\!\!(|t_{\pn,i}\!-\!t_{\pn,j}|)\dfrac{\partial N}{\partial u}(t_{\pn,i},t_{\pn,j},\varphi_{\!\!_n}^{[k]}(t_{\pn,j}))\big(\varphi_{\!\!_n}^{[k]}(t_{\pn,i})\!-\!\varphi_{\!\!_n}^{[k]}(t_{\pn,j})\big).
\end{eqnarray}
Once this system is solved, the coordinates of $\msf{w}_n^{[k+1]}$ allow equation \eqref{Approx-NK-Eq} to become a natural interpolation formula to recover $\varphi_{\!\!_n}^{[k+1]}$ as a function of $s\in\lbr a\,,b\rbr$ as follows: 

\medskip

\noindent
Since for all $s\in\lbr a\,,b\rbr$, 
$$
Q_n^{[k]}(s):=\sum\limits_{j=1}^{\pn}\!w_{\pn,j}g_{_{\delta_{_{n}}}}\!\!(|s\!-\!t_{\pn,j}|)\dfrac{\partial N}{\partial u}(s,t_{\pn,j},\varphi_{\!\!_n}^{[k]}(t_{\pn,j}))
$$
is an approximation of 
$$
I_{n}^{[k]}(s):=\integ{a}{b}{\!\!\!\!g(|s\!-\!t|)\dfrac{\partial N}{\partial u}(s,t,\varphi_{\!\!_n}^{[k]}(t))}{t},
$$
we may suppose that
$$
|I_{n}^{[k]}(s)-Q_n^{[k]}(s)|<\textstyle{\frac{1}{2}}.
$$
Recall that $\varphi_{\!\!_n}^{[k+1]}$ satisfies
\begin{eqnarray*}
\varphi_{\!\!_n}^{[k+1]}(s)&=&\sum\limits_{j=1}^{\pn}\!w_{\pn,j}g_{_{\delta_{_{n}}}}\!\!(|s\!-\!t_{\pn,j}|)\dfrac{\partial N}{\partial u}(s,t_{\pn,j},\varphi_{\!\!_n}^{[k]}(t_{\pn,j}))\msf{w}_n^{[k+1]}(j)+z_n^{[k]}(s)\\\\
&&+\;\varphi_{\!\!_n}^{[k+1]}(s)(I_{n}^{[k]}(s)-Q_n^{[k]}(s)).
\end{eqnarray*}
Hence
\begin{eqnarray*}
\varphi_{\!\!_n}^{[k+1]}(s)&\!\!=\!\!&\dfrac{\sum\limits_{j=1}^{\pn}\!w_{\pn,j}g_{_{\delta_{_{n}}}}\!\!(|s\!-\!t_{\pn,j}|)\dfrac{\partial N}{\partial u}(s,t_{\pn,j},\varphi_{\!\!_n}^{[k]}(t_{\pn,j}))\msf{w}_n^{[k+1]}(j)+z_n^{[k]}(s)}{1\!-\!I_{n}^{[k]}(s)+Q_n^{[k]}(s)},
\end{eqnarray*}
where the denominator satisfies
$$
|1-I_{n}^{[k]}(s)+Q_n^{[k]}(s)|>\textstyle{\frac{1}{2}}\mbox{ for all }s\in\lbr a\,,b\rbr, 
$$
so it never vanishes.

\medskip

\noindent
Comparing \eqref{Ank-classical} with \eqref{Ank-newapp}, \eqref{Dnk-classical} with \eqref{Dnk-newapp}, and \eqref{bnk-classical} with \eqref{bnk-newapp}, we see that, if it happens that $\msf{w}_n^{[k]}=\msf{x}_n^{[k]}$, then $\msf{C}_n^{[k]}=\msf{A}_n^{[k]}$, but even in such a case, neither $\msf{D}_n^{[k]}$ is necessarily equal to $\msf{B}_n^{[k]}$, nor $\msf{b}_n^{[k]}$ necessarily equal to $\msf{a}_n^{[k]}$. 

\medskip

\noindent
This means that the sequences $\big(\msf{x}_n^{[k]}\big)_{k\ge0}$ and $\big(\msf{w}_n^{[k]}\big)_{k\ge0}$ are not necessarily the same, even if the starting points are chosen to be equal: $\msf{x}_n^{[0]}=\msf{w}_n^{[0]}$. 

\medskip 

\noindent
Most probably, we are producing two different numerical approximations of some solution of equation \eqref{Equation}. 

\medskip

\noindent
For the sake of brevity, write:
\begin{eqnarray*}
R_n(x):=\F_n^{\,\prime}(x)^{-1}=(I-T_n(x))^{-1}
\end{eqnarray*}
whenever it exists. 

\begin{thm}[On the convergence of the New Approach] \label{thm:newapp}\quad{}\newline
Suppose that 
\begin{itemize}
\item[] \mbox{\rm(H1)} $I-T(\varphi)$ is bicontinuous,
\item[] \mbox{\rm(H2)} $T$ is Lipschitz-continuous in a neighborhood of $\varphi$,
\item[] \mbox{\rm(H3)} $T_n(x)\stackrel{\nu}\to T(x)$ for all $x$ close enough to $\varphi$.
\end{itemize}
Then, for all $n$ large enough, and $\varphi_{\!\!_{n}}^{[0]}$ close enough to $\varphi$,
$$
\lim\limits_{k\to+\infty}\varphi_{\!\!_{n}}^{[k]}=\varphi.
$$ 
\end{thm}
\begin{proof}
Let be $r_{_{\!0}}>0$ small enough so that
$$
\sup\{\,\|R_n(x)\|~:~x\in\C(\varphi\,,r_{_{\!0}})\,,n\ge2\,\}<+\infty,
$$
and $\partial N/\partial u$ is Lipschitz-continuous on $\C(\varphi\,,r_{_{\!0}})$. Then both $T$ and $T_n$ are  Lipschitz-continuous on $\C(\varphi\,,r_{_{\!0}})$.

\medskip

\noindent
Since $T_n(\varphi)\stackrel{\mathrm{p}}\to T(\varphi)$,
$$
\sup\limits_{n}\|T_n(\varphi)\|<+\infty.
$$
Now,
\begin{eqnarray*}
\varphi_{\!\!_{n}}^{[k+1]}\!-\!\varphi&\!\!=\!\!&\varphi_{\!\!_{n}}^{[k]}\!-\!\varphi\!-\!R_n(\varphi_{\!\!_{n}}^{[k]})\big(\F(\varphi_{\!\!_{n}}^{[k]})\!-\!\F(\varphi)\big)\\
&\!\!=\!\!&\varphi_{\!\!_{n}}^{[k]}\!-\!\varphi\!-\!R_n(\varphi_{\!\!_{n}}^{[k]})\integ{0}{1}{\!\!\!\F^{\,\prime}(\varphi\!+\!t(\varphi_{\!\!_{n}}^{[k]}\!-\!\varphi))}{t}\,(\varphi_{\!\!_{n}}^{[k]}\!-\!\varphi),
\end{eqnarray*} 
i.e.
\begin{eqnarray}\label{rho}
\varphi_{\!\!_{n}}^{[k+1]}\!-\!\varphi&\!\!=\!\!&(I-L_{n}^{[k]})(\varphi_{\!\!_{n}}^{[k]}\!-\!\varphi),
\end{eqnarray}	
where
$$
L_{n}^{[k]}:=R_n(\varphi_{\!\!_{n}}^{[k]})\integ{0}{1}{\!\!\!\F^{\,\prime}(\varphi\!+\!t(\varphi_{\!\!_{n}}^{[k]}\!-\!\varphi))}{t}.
$$
A sufficient condition for the sequence $\big(\varphi_{\!\!_{n}}^{[k]}\big)_{k\ge0}$ to be convergent with limit $\varphi$ is that the spectral radius of $I-L_{n}^{[k]}$ be uniformly bounded by some constant $\gamma<1$. 

\medskip

\noindent
This is indeed the case for all large enough integers $n$, as we prove it now. Remark that
\begin{eqnarray*}
L_{n}^{[k]}&\!\!=\!\!&R_n(\varphi_{\!\!_{n}}^{[k]})\integ{0}{1}{\!\!\!\big(I\!-\!T(\varphi\!+\!t(\varphi_{\!\!_{n}}^{[k]}\!-\!\varphi))\big)}{t}\\
&\!\!=\!\!&I\!+\!R_n(\varphi_{\!\!_{n}}^{[k]})\integ{0}{1}{\!\!\!\big(T_n(\varphi_{\!\!_{n}}^{[k]})-T(\varphi\!+\!t(\varphi_{\!\!_{n}}^{[k]}\!-\!\varphi))\big)}{t}.
\end{eqnarray*}
Hence
\begin{eqnarray*}
I-L_{n}^{[k]}&\!\!=\!\!&R_n(\varphi_{\!\!_{n}}^{[k]})\integ{0}{1}{\!\!\!\big(T(\varphi\!+\!t(\varphi_{\!\!_{n}}^{[k]}\!-\!\varphi))-T_n(\varphi_{\!\!_{n}}^{[k]})\big)}{t}\\\\
&\!\!=\!\!&A_{n}^{[k]}+B_{n}^{[k]}+C_{n}^{[k]}+D_{n},
\end{eqnarray*}
where
\begin{eqnarray*}
A_{n}^{[k]}&\!\!:=\!\!&R_n(\varphi_{\!\!_{n}}^{[k]})\integ{0}{1}{\!\!\!\big(T(\varphi\!+\!t(\varphi_{\!\!_{n}}^{[k]}\!-\!\varphi))-T(\varphi)\big)}{t},\\\\
B_{n}^{[k]}&\!\!:=\!\!&\big(R_n(\varphi_{\!\!_{n}}^{[k]})-R_n(\varphi)\big)(T(\varphi)-T_n(\varphi)),\\\\ 
C_{n}^{[k]}&\!\!:=\!\!&R_n(\varphi_{\!\!_{n}}^{[k]})(T_n(\varphi)-T_n(\varphi_{\!\!_{n}}^{[k]})),\\\\
D_{n}&\!\!:=\!\!&R_n(\varphi)(T(\varphi)-T_n(\varphi)).
\end{eqnarray*}
The Second Resolvent Identity:
$$
R_n(u)-R_n(v)=R_n(u)(T_n(u)-T_n(v))R_n(v)\mbox{ for all }u,v\in\X,
$$
and induction, lead to the following upper bounds. Assume we have chosen $\varphi_{\!\!_{n}}^{[0]}\in\C(\varphi\,,r_{_{\!0}})$, and that, for some integer $k$, $\varphi_{\!\!_{n}}^{[k]}\in\C(\varphi\,,r_{_{\!0}})$. Then, for all $n$ large enough but fixed, there exist constants $\mu_A>0$, $\mu_B>0$ and $\mu_C>0$ such that 
\begin{eqnarray*}
\|A_{n}^{[k]}\|&\!\!\le\!\!&\mu_A\|\varphi_{\!\!_{n}}^{[k]}-\varphi\|,\\
\|B_{n}^{[k]}\|&\!\!\le\!\!&\mu_B\|\varphi_{\!\!_{n}}^{[k]}-\varphi\|,\\
\|C_{n}^{[k]}\|&\!\!\le\!\!&\mu_C\|\varphi_{\!\!_{n}}^{[k]}-\varphi\|,
\end{eqnarray*}
and
$$
\|D_{n}^2\|<\textstyle{\frac{1}{12}},
$$
since
$$
\lim\limits_{n\to+\infty}\|D_n^2\|=0,
$$
and 
$$
\sup\limits_{n}\|D_n\|<+\infty.
$$
Hence, there are constants $c_1>0$, $c_2>0$ such that,
\begin{eqnarray*}
\|(I-L_{n}^{[k]})^2\|&\!\!\le\!\!&\textstyle{\frac{1}{12}}+c_1\,\|\varphi_{\!\!_{n}}^{[k]}-\varphi\|+c_2\,\|\varphi_{\!\!_{n}}^{[k]}-\varphi\|^2.
\end{eqnarray*}
Let be
$$
r_{_{\!1}}:=\min\big\{r_{_{\!0}}\,,\textstyle{\frac{1}{12\,c_1}}\,,\textstyle{\frac{1}{\sqrt{12\,c_2}}}\}.
$$
Assume that $\varphi_{\!\!_{n}}^{[0]}\in\C(\varphi\,,r_{_{\!1}})$, and that, for some integer $k$, $\varphi_{\!\!_{n}}^{[k]}\in\C(\varphi\,,r_{_{\!1}})$ too. Then
$$
\|(I-L_{n}^{[k]})^2\|\le\textstyle{\frac{1}{4}}.
$$
Let $\rho$ denote the spectral radius. Then 
$$
\rho(I-L_{n}^{[k]})=\inf\limits_{m\ge1}\|(I-L_{n}^{[k]})^m\|^{\mbox{\tiny$\frac{1}{m}$}}\le\|(I-L_{n}^{[k]})^2\|^{\mbox{\tiny$\frac{1}{2}$}}\le\textstyle{\frac{1}{2}}.
$$
By \eqref{rho}, we conclude the existence of $\gamma<1$ verifying
$$
\|\varphi_{\!\!_{n}}^{[k+1]}-\varphi\|\le\gamma\|\varphi_{\!\!_{n}}^{[k]}-\varphi\|<r_{_{\!1}}.
$$
This shows that $\varphi_{\!\!_{n}}^{[k+1]}\in\C(\varphi\,,r_{_{\!1}})$.
Finally,
$$
\|\varphi_{\!\!_{n}}^{[k]}-\varphi\|\le\gamma^{\,k}\|\varphi_{\!\!_{n}}^{[0]}-\varphi\|\mbox{ for all }k\ge0.
$$
Hence the sequence $\big(\varphi_{\!\!_{n}}^{[k]}\big)_{k\ge0}$ is convergent with limit $\varphi$.
\end{proof}

\section{Numerical Examples}\label{sec:numerics}  

Numerical experiments were performed with MATLAB\textsuperscript{\textregistered} version 9.8 and Octave version 6.4.0.

\medskip

\noindent
Two examples of nonlinear weakly singular integral operators will  be shown, illustrating the behavior of the methods described in this paper. They differ in the nonlinear factor $N$ and in the weakly singular kernel $g$. 

\medskip

\noindent
In the first example, integrals that should be computed analytically must be approximated because one has no access to a primitive in closed form. Hence, they will be approximated by some numerical quadrature formula with $P_{n}$ nodes, specially conceived for the computation of a weakly singular integral $\integ{a}{b}{\!\!f(t)}{t}$:
\begin{eqnarray*}
\I_{n}(f):=\sum\limits_{\ell=1}^{P_{n}}\rho_{_{P_{n},\ell}}f_{_{\!\mu_{_{n}}}}(\tau_{_{P_{n},\ell}}),
\end{eqnarray*}
where $f_{_{\!\mu_{_{n}}}}$ is a continuous approximation of $f$ defined by $\mu_n$-truncation. 

\medskip

\noindent
Formulas $\I_{n}$ and $\Q_{_{n}\!}$ need not belong to the same family, and grids $\big(t_{\pn,j}\big)_{j=1}^{\pn}$ and  $\big(\tau_{_{P_{n},\ell}}\big)_{\ell=1}^{P_{n}}$ need not be nested.

\medskip

\noindent
$\I_{n}(f)$ must be a significantly better approximation of $\integ{a}{b}{\!\!\!f(t)}{t}$ than $\Q_{_{n}\!}(f)$ since it will be used to compute integrals that should be evaluated exactly. 

\medskip

\noindent
In the second example, the problem is reset in an invariant one-dimensional subspace and the integrals involved in computations are known exactly.

\medskip

\noindent
We recall that in the case of a linear bounded bicontinuous operator $L:\X\to\X$, the problem 
$$
\mbox{For }y\in\X,\mbox{ find }\varphi\in\X\mbox{ such that } L\varphi-y=0,
$$
has a condition number $\kappa$ defined by 
$$
\kappa:=\|L\|\|L^{-1}\|.
$$
If $y\neq0$ then $\varphi\neq0$, and given an approximation $\widehat{\varphi}$ of $\varphi$, its relative error and its relative residual are defined by
$$
e:=\dfrac{\|\widehat{\varphi}-\varphi\|}{\|\varphi\|},\quad
r:=\dfrac{\|L\widehat{\varphi}-y\|}{\|y\|},
$$
respectively. Moreover, $\kappa$, $e$ and $r$ satisfy the inequality
\begin{eqnarray}\label{cond-num-ineq}
\kappa\ge\max\{e/r\,,r/e\}.
\end{eqnarray}
Using the Mean Value Theorem for Derivatives, and the Inverse Function Theorem, the condition number for the nonlinear problem (\ref{Equation}) in a vicinity $\V$ of an exact solution $\varphi\neq0$ appears to be
\begin{eqnarray*}
\kappa_{_\V}(\F)&\!\!\!:=\!\!\!&\sup\limits_{x\in\V}{||\F^{\,\prime}(x)||}\,\sup\limits_{x\in\V}{|| \F^{\,\prime}(x)^{-1}||}.
\end{eqnarray*}
Moreover, in order to keep the inequality (\ref{cond-num-ineq}), the relative residual of an approximate solution $\widehat{\varphi}$ must be defined by
\begin{eqnarray*}
r&\!\!\!:=\!\!\!&\dfrac{\|\F(\widehat{\varphi})\|}{\|\F(0)\|}.
\end{eqnarray*}
For the grids considered in this paper, $\mathbf{e}$ and $\mathbf{r}$ denote the grid-valued relative error and the grid-valued relative residual, respectively. The bound \eqref{cond-num-ineq} is the reason why the ratios $\mathbf{e}/\mathbf{r}$ and $\mathbf{r}/\mathbf{e}$ are shown in the tables of numerical results.

\medskip

\noindent
In both examples:
\begin{itemize}
\item $a:=0$ and $b:=1$.
\item The exact solution is a constant function $\varphi:=c$.
\item The initial point for iterations is the null function.
\item Tables and figures show the convergence process up to the fifth iteration.
\end{itemize}

\subsection{Example 1}\label{subsec:1}

Problem \eqref{Equation} is solved with the Hammerstein operator $K$ defined with
$$
N(s,t,u):=\dfrac{\cos 2 \pi u}{1+s+t+u^4}\mbox{ for all }(s,t,u)\in\lbr0\,,1\rbr {\times}\lbr0\,,1\rbr {\times}\RR,
$$
and with the weakly singular decreasing function $g$ defined by 
$$
g(r):=\dfrac{1}{2\sqrt{r}}\mbox{ for all }r\in\rbr 0\,,1\rbr.
$$
The exact solution $\varphi$ is chosen so that we can assess the quality of the computed approximations. Here it is chosen as a constant 
$$
\varphi(s)= 7\mbox{ for all }s\in\lbr0\,,1\rbr,
$$
yielding a function $y$ that must be approximated numerically, say by truncation followed by the fine numerical quadrature $\I_{n}$.

\medskip

\noindent
The numerical choices for truncation with the New Approach are
$$
\begin{array}{rcl}
\delta_{_{n}}&\!\!:=\!\!&0.00002,\\\\
g_{_{\delta_{_{n}}}}\!(r)&\!\!:=\!\!&\left\{
\begin{array}{ll}
g(\delta_{_{n}})&\mbox{ for all }r\in\lbr0\,,\delta_{_{n}}\rbr,\\
g(r)&\mbox{ for all }r\in\lbr\delta_{_{n}}\,,1\rbr.
\end{array}\right.
\end{array}
$$
The numerical parameters for quadrature with the New Approach are 
\begin{eqnarray*}
n &\!\!:=\!\!& 51,\\
\pn&\!\!:=\!\!& 50,\\
t_{\pn,j}&\!\!:=\!\!&\dfrac{j-1}{50}\mbox{ for all }j\in\ilbr1\,,50\irbr,\\
w_{\pn,j}&\!\!:=\!\!&\dfrac{1}{50},\\
\Q_n&\!\!:=\!\!& \mbox{ Midpoint rectangles}.
\end{eqnarray*}
The integrals to be computed very accurately are approximated with the following parameters
$$
P_{n}:=500,\;\I_{n}:= \Q_{501},\;\mu_n:=0.000002,\;\rho_{_{P_{n},\ell}}=\dfrac{1}{500},\;\tau_{\ell}=\dfrac{\ell-0.5}{500}\mbox{ for all }\ell\in\ilbr1\,,500\irbr.
$$
The relative error and the relative residual with the Classical Approach and $\pn=200$ are shown in Table $\ref{t1}$. 

\medskip

\noindent
The relative error and the relative residual with the New Approach and $\pn=50$ are shown in Table $\ref{t2}$.

\medskip

\noindent
We remark the superiority of the New Approach. Nevertheless, the New Approach cannot keep its superlinear convergence after the precision of the fine quadrature used for the evaluations of the function $y$ is attained, and for which the truncation parameter is $\mu_n:=2{\times}10^{-6}$.

\medskip

\noindent
 The singularity of $g$ in Example 1 is stronger than that of $g$ in Example 2. Still the New Approach converges and is more efficient than the Classical Approach. 
\begin{center}
\begin{table}[ht!]
\centering
\begin{tabular}{|c|c|c|c|c|c|c|}
\hline
	&                   &                       &                             &                   &                   &          \\
$k$ &$\mathbf{r}$       & $\log_{10}\mathbf{r}$ &${\Delta\log_{10}\mathbf{r}}$&$\mathbf{e}$       &${\mathbf{e}/\mathbf{r}}$ & ${\mathbf{r}/\mathbf{e}}$   \\
	&                   &                       &                             &                   &                   &          \\
\hline
	&                   &                       &                             &                   &                   &        \\  
0   & $1\times 10^0$ \; &\;0.0                  &                             & $1\times 10^0$    & 1.0               & 1.0    \\
	&                   &                       & -0.3                        &                   &                   &        \\
1   & $5\times 10^{-1}$ & -0.3                  &                             & $4\times 10^{-1}$ & 1.0               & 1.0    \\  
	&                   &                       & -1.3                        &                   &                   &        \\
2   & $2\times 10^{-2}$ & -1.6                  &                             & $2\times 10^{-2}$ & 1.0               & 1.0    \\
	&                   &                       & -2.3                        &                   &                   &        \\
3   & $1\times 10^{-4}$ & -3.9                  &                             & $3\times 10^{-4}$ & 3.0               & 0.3    \\
	&                   &                       & -0.3                        &                   &                   &        \\
4   & $7\times 10^{-5}$ & -4.2                  &                             & $3\times 10^{-4}$ & 4.0               & 0.3    \\
	&                   &                       &\,0.0                        &                   &                   &        \\
5   & $7\times 10^{-5}$ & -4.2                  &                             & $3\times 10^{-4}$ & 4.0               & 0.3    \\
    &                   &                       &                             &                   &                   &        \\	
\hline
\end{tabular}
\caption{Convergence results for Example $1$ with the Classical Approach and $\pn=200$}

\medskip

\label{t1} 
\end{table}

\begin{table}[ht!]
\centering
\begin{tabular}{|c|c|c|c|c|c|c|}
\hline
	&                   &                       &                             &                   &                   &          \\
$k$ &$\mathbf{r}$       & $\log_{10}\mathbf{r}$ &${\Delta\log_{10}\mathbf{r}}$&$\mathbf{e}$       &${\mathbf{e}/\mathbf{r}}$ & ${\mathbf{r}/\mathbf{e}}$   \\
	&                   &                       &                             &                   &                   &          \\
\hline
	&                   &                       &                             &                   &                   &        \\  
0   & $1\times 10^0$ \; &\;0.0                  &                             & $1\times 10^0$    & 1.0               & 1.0    \\
	&                   &                       & -0.6                        &                   &                   &        \\
1   & $2\times 10^{-1}$ & -0.6                  &                             & $2\times 10^{-1}$ & 0.9               & 1.1    \\  
	&                   &                       & -2.3                        &                   &                   &        \\
2   & $1\times 10^{-3}$ & -2.9                  &                             & $1\times 10^{-3}$ & 1.1               & 0.9    \\
	&                   &                       & -4.1                        &                   &                   &        \\
3   & $1\times 10^{-7}$ & -7.0                  &                             & $2\times 10^{-4}$ & 2492              & 0.0004 \\
    &                   &                       & -1.0                        &                   &                   &        \\
4   & $3\times 10^{-8}$ & -8.0                  &                             & $2\times 10^{-4}$ & 8769              & 0.0001 \\
    &                   &                       &\,0.0                        &                   &                   &        \\
5   & $3\times 10^{-8}$ & -8.0                  &                             & $2\times 10^{-4}$ & 8769              & 0.0001 \\
    &                   &                       &                             &                   &                   &        \\
\hline
\end{tabular} 
\caption{Convergence results for Example $1$ with the New Approach and $\pn=50$}
\label{t2} 

\medskip

\end{table}
\end{center}
The results of the Classical Approach and the New Approach in terms of the evolution of the relative residual are compared in Fig. \ref{fig1}.
\begin{center}
\begin{figure}[ht!]
\centering
\begin{minipage}{0.45\linewidth}
\centering			
\includegraphics[height=6. cm,width=6. cm]{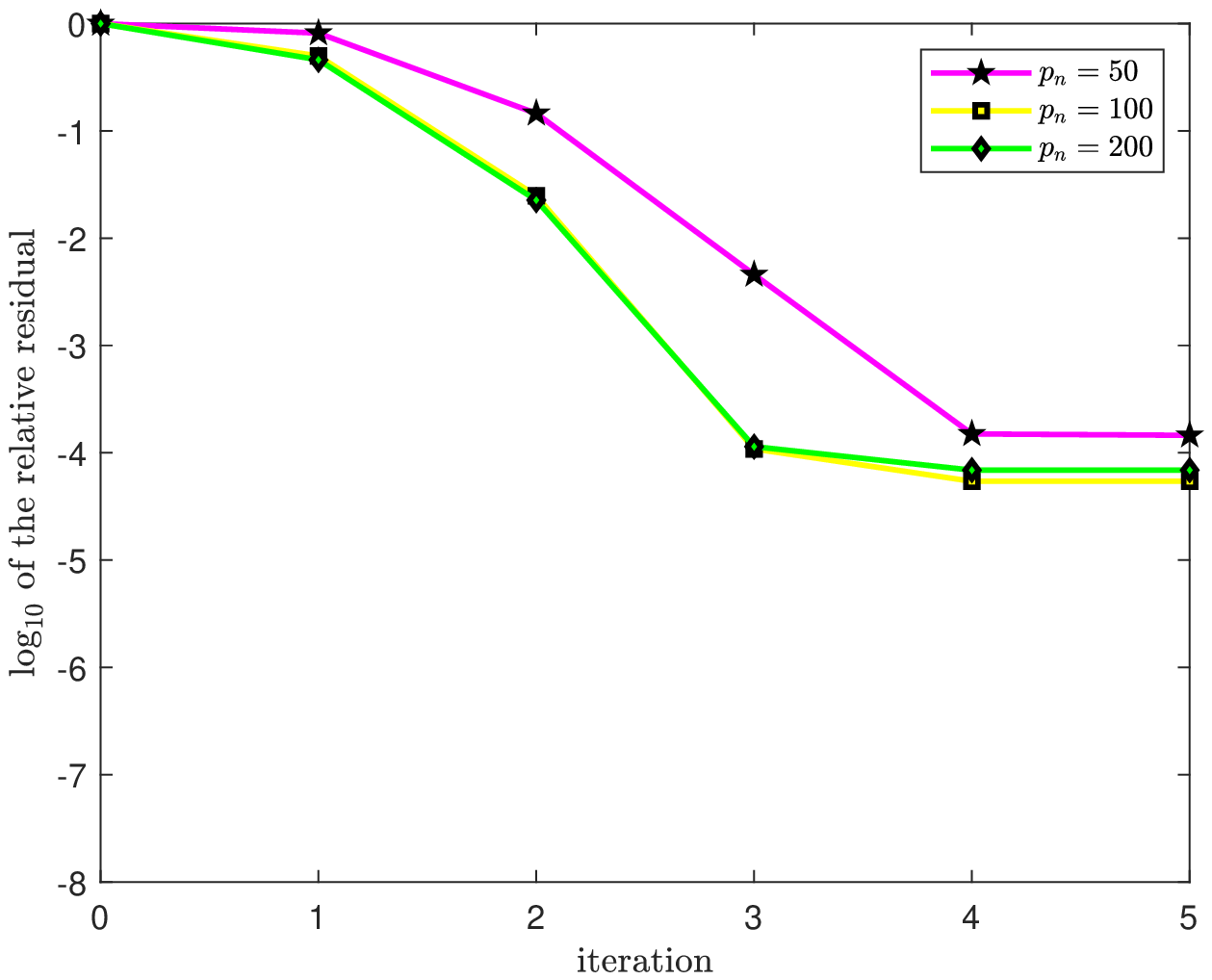}
\centerline{Classical Approach}

\centerline{Discretizing First: $\pn=50\,,100\,,200$}					
\end{minipage}
\begin{minipage}{0.45\linewidth}
\centering	
\includegraphics[height=6. cm,width=6. cm]{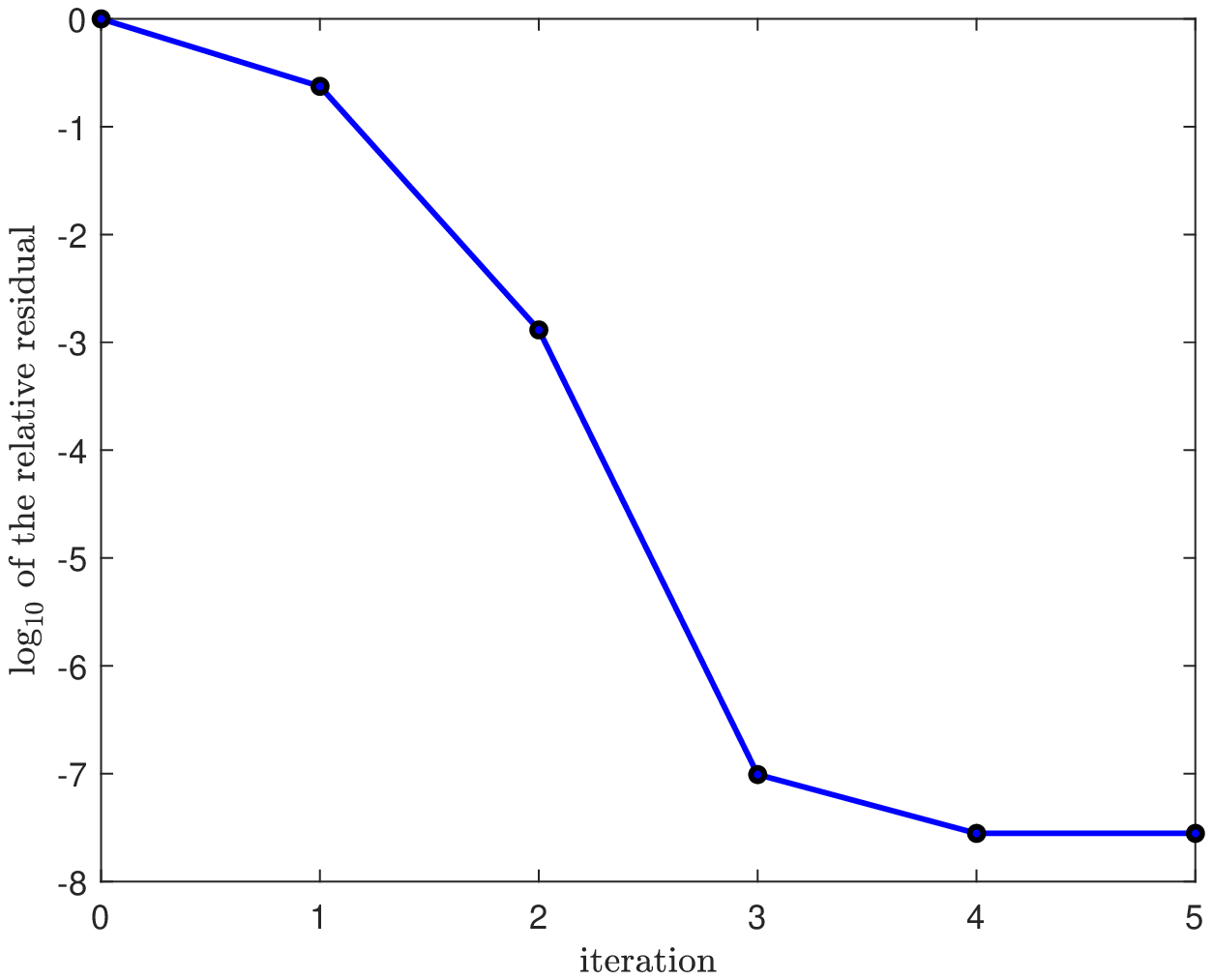}
\centerline{New Approach}

\centerline{Linearizing First: $\pn=50$}			
\end{minipage}
\caption{$\log_{10}$ of the relative residual per iteration in Example 1}
\label{fig1}		
\end{figure}
\end{center}

\subsection{Example 2}\label{subsec:2}

Problem \eqref{Equation} is solved with the Hammerstein operator $K$ defined with 
$$
N(s,t,u):=\dfrac{u}{\log2}+u^3\mbox{ for all }u\in\RR,
$$
and with the weakly singular function $g$ defined by 
$$
g(r):=\log2-\log(1-\cos 2\pi r)\mbox{ for all }r\in\rbr0\,,1\lbr.
$$
The subspace of $\X$ formed by constant functions is invariant under $K$. $K$ is expansive on it. Successive approximations diverge. The New Approach converges at least superlinearly, as shown in Table $\ref{t4}$.

\medskip

\noindent
The constant solution 
$$
\varphi:=-0.5
$$
is associated with the constant function
$$
y:=0.5+0.25\log2.
$$
The numerical choices for truncation with the New Approach are
$$
\begin{array}{rcl}
\delta_{_{n}}&\!\!:=\!\!&0.000001,\\\\
g_{_{\delta_{_{n}}}}\!(r)&\!\!:=\!\!&\left\{
\begin{array}{ll}
g(\delta_{_{n}})&\mbox{ for all }r\in\lbr0\,,\delta_{_{n}}\rbr,\\
g(r)&\mbox{ for all }r\in\rbr\delta_{_{n}}\,,1\!-\!\delta_{_{n}}\lbr,\\
g(\delta_{_{n}})&\mbox{ for all }r\in\lbr1\!-\!\delta_{_{n}}\,,1\rbr.
\end{array}\right.
\end{array}
$$
The numerical parameters for quadrature with the New Approach are 
\begin{eqnarray*}
n&\!\!:=\!\!&101,\\
\pn&\!\!:=\!\!&100,\\
t_{\pn,j} &\!\!:=\!\!&\dfrac{j-1}{100}\mbox{ for all }j\in\ilbr1\,,100\irbr,\\
w_{\pn,j}&\!\!:=\!\!&\dfrac{1}{100},\\
\Q_n&\!\!:=\!\!&\mbox{ Midpoint rectangles}.
\end{eqnarray*}
The integrals to be computed exactly are known analytically since 
$$
\integ{0}{1}{\!\!\!g(|s-t|)}{t}= 2\log2\mbox{ for all }s\in\lbr0\,,1\rbr.
$$
The relative error and the relative residual with the Classical Approach and $\pn=1000$ are shown in Table $\ref{t3}$. 

\medskip

\noindent
The relative error and the relative residual with the New Approach and $\pn=100$ are shown in Table $\ref{t4}$.
\begin{center}
\begin{table}[ht!]
\centering
\begin{tabular}{|c|c|c|c|c|c|c|}
\hline
	&                  &                       &                             &                 &                          &          \\
$k$ &$\mathbf{r}$      & $\log_{10}\mathbf{r}$ &${\Delta\log_{10}\mathbf{r}}$&$\mathbf{e}$     &${\mathbf{e}/\mathbf{r}}$ & ${\mathbf{r}/\mathbf{e}}$   \\
	&                  &                       &                             &                  &                         &          \\
\hline
	&                  &                       &                             &                  &                         &          \\  
0   &$1 \times 10^0$   &\;0.0                  &                             &$1\times 10^0$    & 1.0                     & 1.0\\
	&                  &                       & -0.2                        &                  &                         &          \\
1   &$6\times 10^{-1}$ & -0.2                  &                             &$3\times 10^{-1}$ & 0.5                     & 2.0\\  
	&                  &                       & -0.8                        &                  &                         &          \\
2   &$9\times 10^{-2}$ & -1.0                  &                             &$5\times 10^{-2}$ & 0.6                     & 1.7\\
	&                  &                       & -1.4                        &                  &                         &          \\
3   &$4\times 10^{-3}$ & -2.4                  &                             &$3\times 10^{-3}$ & 0.8                     & 1.3\\
	&                  &                       & -0.3                        &                  &                         &          \\
4   &$2\times 10^{-3}$ & -2.7                  &                             &$1\times 10^{-3}$ & 0.5                     & 2.0\\
	&                  &                       &\,0.0                        &                  &                         &          \\	
5   &$2\times 10^{-3}$ & -2.7                  &                             &$1\times 10^{-3}$ & 0.5                     & 2.0\\
	&                  &                       &                             &                  &                         &    \\	
\hline
\end{tabular}
\caption{Convergence results for Example $2$ with the Classical Approach and $\pn=1000$}
\label{t3} 

\medskip

\end{table}

\begin{table}[ht!]
\centering
\begin{tabular}{|c|c|c|c|c|c|c|}
\hline
	&                  &                       &                             &                 &                          &          \\
$k$ &$\mathbf{r}$      & $\log_{10}\mathbf{r}$ &${\Delta\log_{10}\mathbf{r}}$&$\mathbf{e}$     &${\mathbf{e}/\mathbf{r}}$ & ${\mathbf{r}/\mathbf{e}}$   \\
	&                  &                       &                             &                  &                         &          \\
\hline
	&                  &                       &                             &                  &                         &          \\  
0   &$1 \times 10^0$   &\;0.0                  &                             &$1\times 10^0$    & 1.0                     & 1.0\\
	&                  &                       & -0.2                        &                  &                         &          \\
1   &$6\times 10^{-1}$ & -0.2                  &                             &$3\times 10^{-1}$ & 0.5                     & 2.0\\  
	&                  &                       & -0.9                        &                  &                         &          \\
2   &$8\times 10^{-2}$ & -1.1                  &                             &$5\times 10^{-2}$ & 0.6                     & 1.7\\
	&                  &                       & -1.6                        &                  &                         &          \\
3   &$2\times 10^{-3}$ & -2.7                  &                             &$1\times 10^{-3}$ & 0.5                     & 2.0\\
	&                  &                       & -3.3                        &                  &                         &          \\
4   &$2\times 10^{-6}$ & -6.0                  &                             &$1\times 10^{-6}$ & 0.5                     & 2.0\\
	&                  &                       & -6.0                        &                  &                         &          \\	
5   &$9\times 10^{-13}$&-12.0                  &                             &$6\times 10^{-13}$& 0.7                     & 1.4\\
	&                  &                       &                             &                  &                         &    \\	
\hline
\end{tabular}
\caption{Convergence results for Example $2$ with the New Approach and $\pn=100$}
\label{t4} 

\medskip

\end{table}
\end{center}
The results of the Classical Approach and the New Approach in terms of the evolution of the relative residual are compared in Fig. \ref{fig2}.
\begin{center}
\begin{figure}[ht!]
\centering
\begin{minipage}{0.45\linewidth}
\centering			
\includegraphics[height=6. cm,width=6. cm]{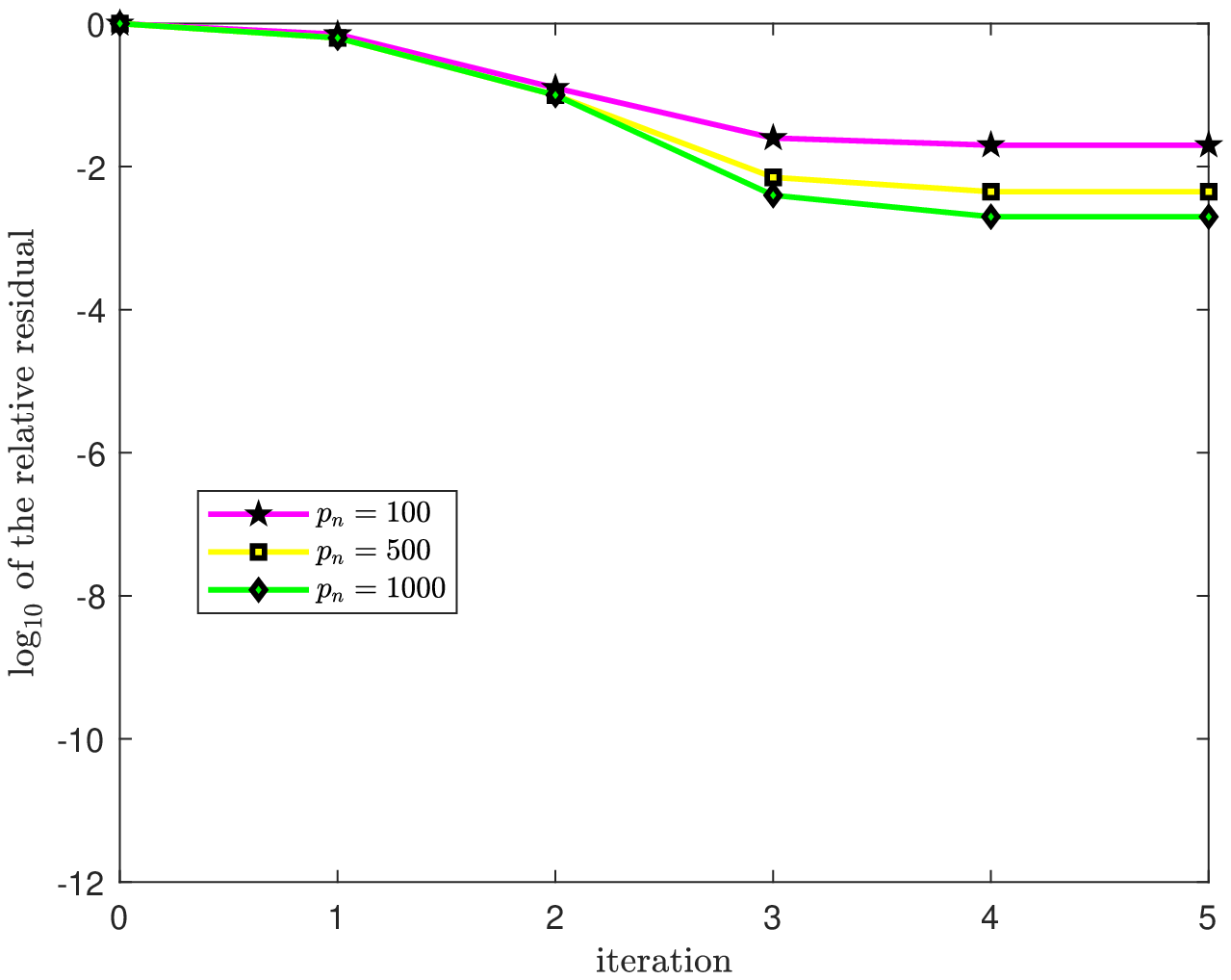}
\centerline{Classical Approach}

\centerline{Discretizing First: $\pn=100\,,500,1000$}
\end{minipage}
\begin{minipage}{0.45\linewidth}
\centering	
\includegraphics[height=6. cm,width=6. cm]{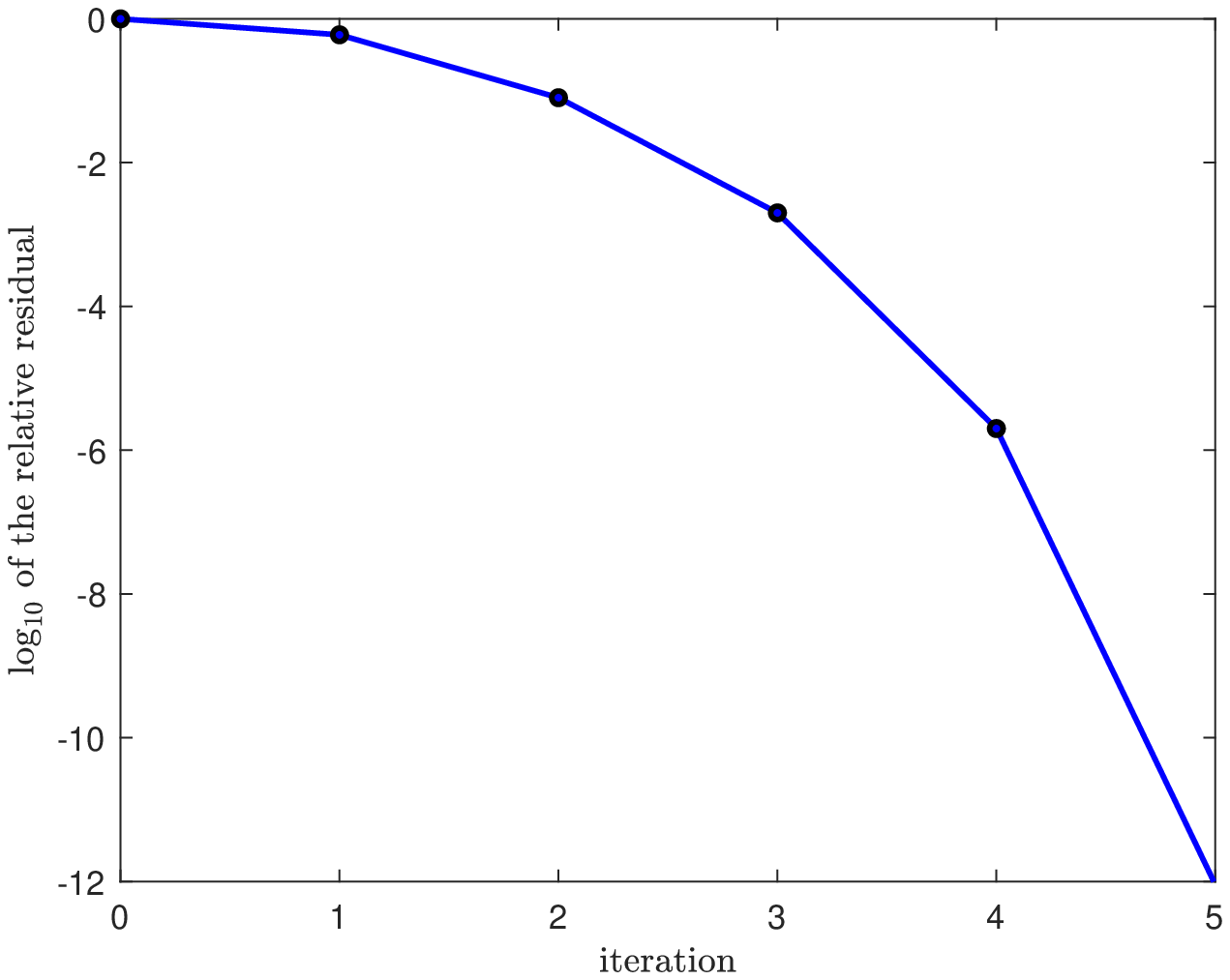}
\centerline{New Approach}

\centerline{Linearizing First: $\pn=100$}
\end{minipage}
\caption{$\log_{10}$ of the relative residual per iteration in Example 2}
\label{fig2}
\vspace*{6.5 mm}
\end{figure}
\end{center}

\section{Final Comments and Conclusions}\label{sec:commconcl}

The classic textbook \cite{Ch1983} by Françoise Chatelin (1941 -- 2020), first published by Academic Press in 1983, provides a unified treatment of linear integral equations of the second kind, and spectral approximation for Fredholm linear integral operators. Despite significant changes and advances in the field since it was first published, the book continues to form the theoretical bedrock for any computational approach to integral equations and spectral theory. Almost all the papers of the authors of this article have been inspired by Chatelin's research and academic activity during the 80's. 

\medskip

\noindent
In this work, we have extended to nonlinear integral operators, the singularity subtraction technique presented in \cite{An1981} for approching  linear weakly singular integral operators in the framework of real valued continuous functions. The singularity subtraction technique cannot be settled in Lebesgue spaces.

\medskip

\noindent
In the Classical Approach, $\varphi$ is approximated by a sequence of functions $\big(\varphi_{_{\!\!_n}}\big)_{n\ge2}$. For a fixed $n$, $\varphi_{_{\!\!_n}}$ is approximated only at the nodes of the grid, with the help of the Newton-Kantorovich method in the $\pn$-dimensional real space $\RR^{\pn{\times}1}$. This method builds the sequence $\big(\msf{x}_n^{[k]}\big)_{k\ge0}$. This sequence approximates the grid values of $\varphi_{_{\!\!_n}}$: $\msf{x}_n^{[k]}(i)=\varphi_{_{\!\!_n}}(t_{\pn,i})$ + \lbr N-K $k$-step error\rbr.

\medskip

\noindent
In the New Approach, the Newton-Kantorovich method is applied in the infinite dimensional space $\X$, the first sequence to appear is $\big(\varphi^{[k]}\big)_{k\ge0}$. Since it cannot be computed exactly, the singularity subtraction approximation is used and a new sequence appears: $\big(\varphi_{\!\!_n}^{[k]}\big)_{k\ge0}$. To compute $\varphi_{\!\!_n}^{[k]}$, a linear system is solved for $\msf{w}_n^{[k]}\in\RR^{\pn{\times}1}$ and gives the exact values of $\varphi_{\!\!_n}^{[k]}$ at the nodes up to the error made by the routine used to solve the system. 

\medskip

\noindent
All three, $\varphi_{_{\!\!_n}}$, $\varphi^{[k]}$ and $\varphi_{\!\!_n}^{[k]}$, are approximations of $\varphi$, although they approximate $\varphi$ in different ways:
\begin{eqnarray*}
\lim\limits_{n\to+\infty}\varphi_{_{\!\!_n}}&=&\varphi,\mbox{ but $\varphi_{_{\!\!_n}}$ is not computable,}\\
\lim\limits_{k\to+\infty}\varphi^{[k]}&=&\varphi,\mbox{ but $\varphi^{[k]}$ is not computable,}\\
\lim\limits_{k\to+\infty}\varphi_{\!\!_n}^{[k]}&=&\varphi\mbox{ for a fixed $n$ {\sl reasonably} large},
\end{eqnarray*}
where $\varphi_{\!\!_n}^{[k]}(s)$ could be known for all $s\in\lbr a\,,b\rbr$, if some involved integrals were calculated exactly in its natural interpolation formula. Summarizing:

\medskip

\noindent
-- Inconvenients of $\varphi_{_{\!\!_n}}$: There is no natural interpolation formula to compute $\varphi_{_{\!\!_n}}(s)$ for all $s\in\lbr a\,,b\rbr$. It is impossible to know its grid value $\varphi_{_{\!\!_n}}(t_{\pn,i})$ exactly and it will be approximated by the last iterate of the Newton-Kantorovich method. To compute the coefficient matrix and the right hand side of the linear system corresponding to each N-K iteration, some integrals must be approximated numerically with a higher order numerical quadrature better than the $n$-dependent approximations involved in the singularity subtraction scheme.

\medskip

\noindent
-- Inconvenients of $\varphi^{[k]}$: It is not computable at all and it must be approximated by $\varphi_{\!\!_n}^{[k]}$, issued from the singularity subtraction scheme.

\medskip

\noindent
-- Inconvenients of $\varphi_{\!\!_n}^{[k]}$: To compute $\varphi_{\!\!_n}^{[k]}(s)$ for a given $s\in\lbr a\,,b\rbr$, some integrals must be computed with a higher order numerical quadrature better than the $n$-dependent approximations involved in computations. Its grid values $\varphi_{\!\!_n}^{[k]}(t_{\pn,i})$ are the solutions $\msf{w}_n^{[k]}(i)$ of a linear system. As before, to compute the coefficient matrix and the right hand side of this system, some integrals must be approximated numerically with a higher order numerical quadrature better than the $n$-dependent approximations involved in the singularity subtraction scheme.

\medskip

\noindent
Since the rate of convergence in \eqref{NK} is at least linear (for low values of $n$) and can be almost quadratic (for reasonably large values of $n$), it is clear that the New Approach is the most intelligent and economic scheme to build an approximation of $\varphi$.

\medskip

\noindent
A major survey on numerical approximation of nonlinear integral equations is \cite{At1992}. This paper studies numerical methods for calculating fixed points of nonlinear integral operators, i.e. equations of the form $\varphi=K(\varphi)$ with the notation of our paper. This corresponds to the case $y=0$ and is less general than the work presented here since $y$ cannot be incorporated as a part of the integral operator $K$. Methods treated in \cite{At1992} include a product integration type scheme for weakly singular Hammerstein operators, projection methods and Nyström methods. As in our paper, all those methods require the solution of finite-dimensional systems of nonlinear equations. An auxiliary numerical method is needed to solve these nonlinear finite-dimensional systems.

\bigskip

\begin{center}
{\bf Acknowledgements}
\end{center}
The second and fourth authors were partially supported by CMUP, which is financed by national funds through FCT – Fundação para a Ciência e Tecnologia, I.P., under the project with reference UIDB/00144/2020.
The research of the third author was partially financed by Portuguese Funds through FCT (Fundação para a Ciência e a Tecnologia) within the Projects UIDB/00013/2020 and UIDP/00013/2020. 



\begin{thebibliography}{99}

\bibitem{AhLaLi2001} M. Ahues, A. Largillier and B.~V. Limaye: Spectral Computations for Bounded Operators, Chapman \& Hall/CRC, Boca Raton, FL (2001).

\bibitem{AhDAFeVa2019} M. Ahues, F.~D. D'Almeida, R. Fernandes and P. Vasconcelos: Singularity Subtraction for Nonlinear Weakly Singular Integral Equations of the Second Kind, In {\it Integral Methods in Science and Engineering}, Vol. 1. Theoretical Techniques, C. Constanda and Paul Harris Editors, Birkhäuser Verlag, New York, 1-13 (2019)

\bibitem{An1971} P. Anselone: Collectively compact operator approximation theory and applications to integral equations, Prentice-Hall, Englewoodcliffs, NJ (1971).

\bibitem{An1981} P. Anselone: Singularity subtraction in the numerical solution of integral equations, J. Austral. Math. Soc. Ser. B, 22, 408-418 (1981).

\bibitem{At1992} K. Atkinson: A survey of numerical methods for solving nonlinear integral equations, Journal of Integral Equations, 4, 1, 15-46 (1992).

\bibitem{Ch1983} F. Chatelin: Spectral Approximation of Linear Operators, Classics in Applied Mathematics, SIAM (2011).

\bibitem{GrAhDA2014} L. Grammont, M. Ahues and F. D'Almeida: For nonlinear infinite dimensional equations, which to begin with: linearization or discretization?, Journal of Integral Equations and Applications, Vol. 26, 3, 413-436 (2014).

\bibitem{GrVaAh2016} L. Grammont, P. Vasconcelos and M. Ahues: A modified iterated projection method adapted to a nonlinear integral equation, J. Appl. Math. Comput. Vol. 276 pp 432-441 (2016).

\bibitem{Sh2008} L.~F. Shampine: Vectorized Adaptive Quadrature in MATLAB, J. Comput. Appl. Math., 211, 131-140 (2008).

\bibitem{XiBo2012} S. Xiang and F. Bornemann: On the Convergence Rates of Gauss and Clenshaw-Curtis Quadrature for Functions of Limited Regularity, SIAM J. on Numer. Anal., 50, 5, 2581-2587 (2012).

\end{thebibliography}
\end{document}